\documentclass[10pt]{article}
\usepackage{stmaryrd}
\usepackage{textcomp}
\usepackage{enumerate}
\usepackage{setspace}
\usepackage{amssymb}
\usepackage{amsthm}
\usepackage{amsmath}
\usepackage{graphicx}
\usepackage{extarrows}
\usepackage{marvosym}
\usepackage{empheq}
\usepackage{latexsym}
\usepackage{endnotes}
\usepackage{fontenc}
\usepackage{color}

\begin{document}

	\newtheorem{theorem}{Theorem}
	\newtheorem{definition}[theorem]{Definition}
	\newtheorem{question}[theorem]{Question}
	\newtheorem{conjecture}[theorem]{Conjecture}
	\newtheorem{remark}[theorem]{Remark}
	
	\newtheorem{lemma}[theorem]{Lemma}
	\newtheorem{claim}[theorem]{Claim}
	\newtheorem*{claim*}{Claim}
	\newtheorem{corollary}[theorem]{Corollary}
	\newtheorem{problem}[theorem]{Problem}
	\newtheorem{fact}[theorem]{Fact}
	\newtheorem{observation}[theorem]{Observation}
	\newtheorem{prop}[theorem]{Proposition}

	\theoremstyle{remark}
	\newtheorem{example}[theorem]{Example}
	
	\def\eps{\varepsilon}
	\def\HH{\mathcal{H}}
	\def\EE{\mathbb{E}}
	\def\C{\mathbb{C}}
	\def\R{\mathbb{R}}
	\def\Z{\mathbb{Z}}
	\def\PP{\mathbb{P}}
	\def\l{\lambda}
	\def\s{\sigma}
	\def\t{\theta}
	\def\z{\zeta}
	\def\endproof{{\hfill $\square$} }
	\def\Xt{\widetilde{X}}
	\def\Pt{\widetilde{P}}
	\def\Var{\mathrm{Var}}
	\def\PV{\mathrm{PV}}
	\def\NV{\mathrm{NV}}
	\def\N{\mathcal{N}}
	\def\jj{\mathbf{j}}
	
	\def\tm{\tilde{\mu}}
	\def\ts{\tilde{\sigma}}
	\newcommand\marginal[1]{\marginpar{\raggedright\parindent=0pt\tiny #1}}
	\def\Pois{\mathrm{Pois}}
	\def\Bin{\mathrm{Bin}}
	\reversemarginpar

	\title{Central limit theorems from the roots of probability generating functions}
	\author{Marcus Michelen\footnote{University of Pennsylvania, Department of Mathematics, David Rittenhouse Lab. 209 South 33rd Street
Philadelphia, PA 19104-6395} \ 
	and 
	Julian Sahasrabudhe\footnote{Instituto Nacional de Matem\'{a}tica Pura e Aplicada (IMPA), 
Estr. Dona Castorina, 110 Jardim Botânico, Rio de Janeiro RJ 22460-320, Brasil \emph{and}
Peterhouse, University of Cambridge, Trumpington Steet, Cambridge, UK, CB2 1RD.}}
	\date{}
	
	\maketitle

\begin{abstract}
For each $n$, let $X_n \in \{0,\ldots,n\}$  be a random variable with mean $\mu_n$, standard deviation $\s_n$, and let
\[ P_n(z) = \sum_{k=0}^n \PP( X_n = k) z^k ,\] be its probability generating function. We show that if none of the complex zeros of the polynomials $\{ P_n(z)\}$ is contained in a neighbourhood of $1 \in \mathbb{C}$ and $\s_n > n^{\eps}$ for some $\eps >0$, then 
$ X_n^* =(X_n - \mu_n)\s^{-1}_n$ is asymptotically normal as $n\rightarrow \infty$: that is tends in distribution to a random variable $Z \sim \mathcal{N}(0,1)$. Moreover, we show this result is sharp in the sense that there exist sequences of random variables $\{X_n\}$ with $\s_n > C\log n $ for which $P_n(z)$ has no roots near $1$ and $X_n^*$ is not asymptotically normal. These results disprove a conjecture of Pemantle and improve upon various results in the literature.  We go on to prove several other results connecting the location of the zeros of $P_n(z)$ and the distribution of the random variable $X_n$.
 
\end{abstract}		
	
	\section{Introduction}		
	
Let $X \in \{0,\ldots,n\}$ be a random variable with mean $\mu_n$ and standard deviation $\s_n$ and let the polynomial $P_X(z)$ be its \emph{probability generating function}
\[ P_X(z) = \sum_{k=0}^n \PP(X = k )z^k.
\] In this paper, we are concerned with the following general question: what can be deduced about the distribution of the random variable $X$ from information on the location of the roots of $P_X(z)$ in the complex plane? For example, suppose that we are given a random variable $X \in \{0,\ldots,n\}$ and know only that the zeros $\z$ of $P_{X}(z)$ are real. With this knowledge at hand, we readily factor $P_X$ as $P_X(z) = c\prod_{i=1}^n(z + \zeta_i)$ (assuming $\PP(X = n)\not=0$ for the moment) and notice that the roots of $P_X$ must be non-positive as the coefficients of $P_X(z)$ are non-negative. Hence, we may write
$P_X(z) = \prod_{i=1}^n (q_iz + 1-q_i)$, for some $q_1,\ldots q_n \in [0,1]$. A little further thought reveals that this special expression for the probability generating function corresponds to an expression of $X$ as a sum of independent random variables $X = X_1 + \cdots + X_n$, where $X_i$ is the $\{0,1\}$-random variable, taking $1$ with probability $q_i$. Thus, with an appropriate central limit theorem at hand, we see that $X$ must be \emph{approximately normal}, provided the variance of $X$ is sufficiently large. In other words, from this one piece of information (albeit a strong piece of information) about the zeros of $P_X(z)$, one can quickly deduce quite a bit of information about the distribution of $X$. The aim of this paper is to show that this assumption of real rootedness of $P_X(z)$ can be related quite considerably while yielding similar results. In particular, we give three different results each of which says that ``if $X$ has large variance and the roots of $P_X(z)$ avoid a region in the complex plane, then $X$ is approximately normal.'' 

\subsection{History}
Before turning to our contributions, we take a brief moment to situate our results in an old and well-studied field centred around the following question: What does information about the coefficients of $P(z)$ tell us about the distribution of the complex roots of $P$ (and vice versa)? 
This question has a long and rich history, reaching back to the seminal work of Littlewood, Szeg\H{o}, P\'{o}lya, and perhaps even Cauchy, due to his 1829 proof \cite{cauchy1828exercises} of the fundamental theorem of algebra, which gives explicit bounds on the magnitude of the complex roots (see \cite[Theorem 1.2.1]{borwein-erdelyi} for a modern treatment of this proof). 

One line of research, initiated by the 1938 - 1943 work of Littlewood and Offord \cite{LittlewoodOfford,LittlewoodOfford-II,LittlewoodOfford-III}, concerns the typical distribution of roots of random polynomials. 
For example Kac \cite{kac1943} gave an exact integral formula for the number of real roots of random polynomial, with coefficients sampled independently from a normal distribution. Later, Erd\H{o}s and Offord \cite{ErdosOfford} showed that as $n\rightarrow \infty$ almost all polynomials of the form $\sum_{i=0}^n \eps_i x_i$, where $\eps_1,\ldots,\eps_n \in \{0,+1,-1\}$, have $(2/\pi + o_n(1))\log n$ real roots. Since these results, numerous other settings have been explored \cite[Example I.2]{kac-book}, including varied models as well as extensions to the roots of random power series \cite{offord,peres-virag}. 

Deterministic results have also received considerable attention. To name a few, Bloch and P\'{o}lya \cite{BlochPolya} studied the maximum number of roots of a polynomial with coefficients in $\{0,-1,+1\}$ and degree $n$. After improvements by Schur \cite{Schur} and Szeg\H{o} \cite{Szego}, this line culminated in the remarkable and celebrated result of Erd\H{o}s and Tur\'{a}n \cite{erdos-turan} from 1950: if a polynomial $P(z) = \sum_{k=0}^n a_kz^k$ has sufficiently ``flat'' coefficients, meaning $(|a_0||a_n|)^{-1/2}\sum_{k} |a_k| = e^{o(n)}$, then the roots of $P(z)$ are approximately ``radially equidistributed'', in the sense that for every $0 < \alpha < \beta < 2\pi$, the number of roots $\z = re^{i\phi}$ with $\alpha \leq \phi \leq \beta$ is $(\beta -\alpha)n/(2\pi) + o(n)$. Further refinements have been obtained by Ganelius \cite{Ganelius}, Mignotte \cite{Mignotte} and more recently, by Erd\'{e}lyi \cite{Erdelyi} and Soundararajan \cite{Sound}. Also, Bilu \cite{Bilu} obtained a beautiful variant of the Erd\H{o}s-Tur\'{a}n theorem for higher dimensions and with more algebraic restrictions (also see Granville \cite{AGran} for some discussion).

In another direction, Odlyzko and Poonen \cite{OP} have studied the geometric properties of the set of points in the complex plane that are the zero of
\emph{some} polynomial with coefficents in $\{0,1\}$; they show that the closure of this set is path connected, and that it appears to exhibit a certain fractal-like structure. Beaucoup, Borwein, Boyd and Pinner \cite{BorweinPowerSeries}, in a similar vein, have studied the multiple real roots of power-series with restricted coefficients.

Interestingly, and most relevant to our work here, the roots of polynomials with non-negative coefficients are also known to have several particular properties \cite{nonNegFactor,eremenko}. To take an example in the flavour of Erd\H{o}s and Tur\'{a}n, an old observation of Obrechkoff \cite{obrechkoff} says that if $f$ is a degree $n$ polynomial with non-negative coefficients and $\alpha \in [0,\pi]$, then the number of zeros $\z = re^{\phi}$ with $- \alpha < \phi < \alpha $ is \emph{at most} $2\alpha n/\pi$; that is, ``at most twice as much as the equidistributed case.''    

\subsection{Results}

In the present work we take a slightly different perspective from many of the results above; instead of assuming ``small scale'' information about the coefficients (like assuming that they takes values in $\{0,1\}$, say), we look to connect ``large scale'' distributional information of the coefficients with the distribution of the roots. Indeed, we restrict our attention to polynomials with non-negative coeffiecents and think of the polynomial as a probability generating function of a random variable. 

In this line, Hwang and Zacharovas \cite{hwang-zacharovas} showed that if a sequence of random variables $\{X_n\}$ is such that all of the zeros of $\{P_n\}$ lie on the unit circle and $\deg(P_n) \rightarrow \infty$, then the limiting distribution of $X^*_n$ is completely determined by its fourth centralized moments $\EE( X_n -\mu_n)^4$. As a consequence, they gave a simple criterion for $\{X_n^*\}$ to be asymptotically normal. 

Later, Lebowitz, Pittel, Ruelle and Speer (Henceforth LPRS) \cite{LPRS} studied a looser restriction on the roots that guarantees a central limit theorem for $X_n$ with large variance. They showed that if $\s_n n^{-1/3} \rightarrow \infty$, and none of the roots of $\{P_n(z)\}$ is contained in a neighbourhood of $1 \in \mathbb{C}$, then $X_n$ satisfies a central limit theorem. This lead Pemantle \cite{pemantle} to conjecture that a similar result holds with a weaker criterion on the variance. Namely, he conjectured that $\s_n \rightarrow \infty$ is sufficient to guarantee a central limit theorem
if all the roots of $\{P_n(z)\}$ avoid a neighbourhood of $1$. We show that this conjecture is false, however the condition on the variance in the Theorem of LPRS can be considerably improved, as anticipated by Pemantle.

	\begin{theorem} \label{thm:largeVarThm} Let $\eps > 0$ and, for each $n$, let $X_n \in \{0,\ldots,n\}$ be a random variable with standard deviation $\s_n > n^{\eps}$.
		If all the roots $\zeta$ of $P_n$ satisfy $|1 - \zeta| > 1/\sigma_n^{1-\eps}$ then the sequence $\{X_n\}$ satisfies a central limit theorem.

	\end{theorem}
On the other hand, we show that for every $\delta >0$ there exists $\{X_n\}$ with $\s_n > C\log n$, for which $X_n^* \not\to \mathcal{N}(0,1)$ in distribution and $|\z - 1| > 1$ for all roots $\z$ of the polynomials $\{P_n(z)\}$. The obvious question that arises is: what is the correct variance condition in Theorem~\ref{thm:largeVarThm}? It is perhaps reasonable to put forward the following modified version of Pemantle's original conjecture. 

\begin{conjecture} Let $\delta >0$ and, for each $n$, let $X_n \in \{0,\ldots,n\}$ be a random variable with variance $\s_n$. If all the roots $\z$ of $P_n(z)$ satisfy $|1-\z| > \delta$ and $\s_n/(\log n)^c \rightarrow \infty$ for every $c > 0$ then $\{X_n\}$ satisfies a central limit theorem. \end{conjecture}

We note in passing that Theorem~\ref{thm:largeVarThm} also implies an improvement on the work of Ghosh, Liggett and Pemantle \cite{GLP} who considered a similar situation for vector-valued random variables. Let $\{Y_n\}$ be a sequence of random variables taking values in $\{0,\ldots,n\}^d$ and let $P_n(z_1,\ldots,z_d) \in \mathbb{R}[z_1,\ldots,z_d]$ be the corresponding probability generating functions 

\[ P_{n}(z_1,\ldots,z_d) = \sum_{0 \leq j_1,\ldots,j_d \leq n} \PP(Y_n = (j_1,\ldots,j_d))z_1^{j_1} \cdots z_{d}^{j_d}.\]

They showed that if the polynomials $P_n$ are ``real stable'' then $Y_n^*$ tends to a \emph{multivariate} normal, provided the variance grows sufficiently quickly. 
Here, a polynomial $P(x) \in \mathbb{R}[x_1,\ldots,x_d]$ is said to be \emph{real stable} if each of its roots $\z = (\z_1,\ldots,\z_d)$ has at least one coordinate $\z_i$ with imaginary part $\Im(\z_i) \leq 0$. This class of polynomials, admittedly a little strange at first blush, arise naturally in many situations \cite{pemantle-survey}; indeed, the corresponding random variables can be thought of as vast generalizations of determinantal measures (see the discussion in \cite{GLP} or \cite[Theorem 2]{soshnikov} together with \cite[Proposition 3.2]{bbl}). Our Theorem~\ref{thm:largeVarThm} implies an improvement on the variance condition in the theorem of Ghosh, Liggett and Pemantle theorem, and partially answers a question of theirs \cite{GLP}. 

\begin{corollary} \label{cor:GLP} Let $\eps >0$ and for each $n$, let $Y_n = (Y_n^{(1)},\ldots,Y_n^{(k)}) \in \{1,\ldots,n\}^k$ be a random variable with covariance matrix $A_n$ and real stable probability generating function $P_n$. If there exists a sequence of real numbers with $s_n > n^{\eps}$ and a $k \times k$ matrix $A$ for which
$ A_n s_n^{-2} \rightarrow A $ then 
\[ \frac{Y_n - \EE Y_n }{s_n} \rightarrow \mathcal{N}(0,A). \]
\end{corollary}

Our next result says that if all the roots $\z$ of $P_n(z)$ grow polynomially, that is satisfy $|\z| > n^{\delta}$ for some fixed $\delta >0$ as 
$\sigma_n \rightarrow \infty$, then $\{X_n\}$ satisfies a central limit theorem. Actually the proof of this result naturally provides a slightly stronger result.

	\begin{theorem} \label{thm:TlToZero} Let $k > 0$ and, for each $n$, let $X_n \in \{0,\ldots,n\}$ be a random variable with standard deviation $\s_n$. Let $\Lambda_n$ be the set of roots of the probability generating function $P_n$ of $X_n$. 
If $\s_n \rightarrow \infty$ and \[ \sum_{\z \in \Lambda_n} |\z|^{-k} \rightarrow 0 
	\] as $n \rightarrow \infty$ then $\{X_n\}$ satisfies a central limit theorem.
	\end{theorem}

Theorem~\ref{thm:TlToZero} is also best possible, in the sense that for every function $\eps(n)\rightarrow 0$ there exists a sequence of random variables $\{X_n\}$, so that $\s_n \rightarrow \infty$, all of the roots $\z$ of $P_n(z)$ satisfy $|\z| > n^{\eps(n)}$, and $X_n^*$ does not tend to a normal in distribution. 

In light of Theorem~\ref{thm:largeVarThm}, it is natural to ask if there is a larger neighbourhood $R$ of $1$ so that if the roots of $P_n(z)$ avoid $R$ 
then  $\{X_n\}$ satisfies a central limit theorem whenever $\s_n \rightarrow \infty$. Our last theorem shows 
that this is true if we choose $R$ to be a neighbourhood of $1$ along with an open set containing the region
 $S= \left\lbrace x + iy : x \geq \frac{2x^2 + 2y^2}{1 + x^2 + y^2}\right\rbrace$, 
as shown in Figure \ref{fig:diagram}. While this region seems to be a strange choice, it is actually quite natural; if the roots of $P_n(z)$ avoid this region then each of the roots of $P_n(z)$ can be thought of as contributing positively to the variance of $X_n$. Moreover $R$ is the largest region for which this is true.
 
	\begin{figure}[h]
		\centering
		\includegraphics{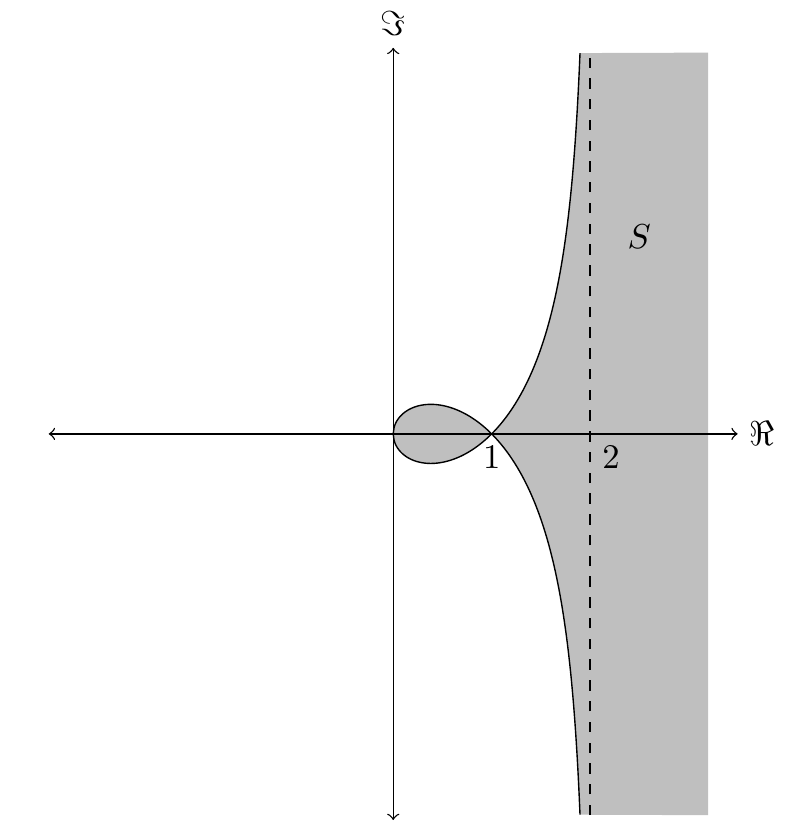}
		\caption{The region $S=\left\{x + iy : x \geq \frac{2x^2 + 2y^2}{1 + x^2 + y^2}\right\}$}
		\label{fig:diagram}
	\end{figure}

To state our theorem, let $\delta > 0$ and let $N_n = N_n(\delta)$ be the number of zeros of $P_n$ with distance at most $\delta$ to $S$, where the roots are counted with multiplicity. We prove the following:
	
	\begin{theorem} \label{thm:region}
		Let $\delta >0$ and let $X_n \in \{0,\ldots,n\}$ be a sequence of random variables with $\sigma_n \rightarrow \infty$. If every zero $\zeta$ of $P_n$ satisfies $|1 - \zeta| > \delta$ and $N_n(\delta) = o(\sigma_n^{3})$ then $\{X_n\}$ satisfies a central limit theorem. 
	\end{theorem}
	
\subsection{Some remarks on the proofs}

To prove Theorems~\ref{thm:largeVarThm} and \ref{thm:TlToZero}, we control the rescaled characteristic functions $\phi_n(\t/\s_n)$. Our first step is to find an appropriate form of this function in terms of the roots of the polynomial $P_n$. As it turns out, it is rather difficult to work directly with $\phi(\t/\s_n)$ so we instead opt to work with $\phi_n(\t/b_n)$ for some appropriately chosen sequence $b_n$. One of our main ingredients is captured in Lemma~\ref{pr:high-cumulants-enough}, which allows us to use these results to relate information about $\phi_n(\t/b_n)$ to $\phi(\t/\s_n)$. To control $\phi(\t/b_n)$ in the proof of Theorem~\ref{thm:largeVarThm}, we carefully analyse how much each root contributes to our exponential representation of $\phi$. The proof of Theorem~\ref{thm:TlToZero} uses a technique from linear algebra. The proof of Theorem~\ref{thm:MThm} is rather different: we use a moment method to show convergence to a normal.

It is perhaps interesting to note that our results seem to use the fact that the polynomials $\{P_n\}$ have non-negative coefficients in a much more central way than in previous work. To understand what is meant by this, let $P$ be an (arbitrary) polynomial with $P(1) = 1$. We may formally define $\mu = P'(1)$ and $\sigma^2 =  P''(1) + P'(1)-(P'(1))^2$ and then say that a sequence of such polynomials $\{P_n(z)\}_n$ satisfies a central limit theorem if $P_n(e^{i \theta/\sigma_n})e^{-i\theta \mu_n / \sigma_n} \to e^{-\theta^2 / 2}$. The results of Hwang and Zacharovas and LPRS both work work in this more general setting and, indeed, in this generalized setting, the bound on the variance \cite{LPRS} cannot be improved. Thus Theorems \ref{thm:largeVarThm} and \ref{thm:TlToZero} depend more deeply on the hypothesis of non-negative coefficients. 


	\section{Some Preparations} \label{sec:Preparations}
	
	 If $X$ is a real-valued random variable, we define, as is standard, the \emph{characteristic function of} $X$ as the function $\phi_X(\theta) = \EE e^{i\t X } $, for $\theta \in \R$. A key ingredient in our results will be the following theorem of Marcinkiewicz \cite[Theorem 7.3.3]{lukacs} \cite{Marcin}, which gives us some information about the structure of characteristic functions that have a specific exponential form.
	
	\begin{theorem} \label{thm:MThm} Let $P(X) \in \mathbb{C}[X]$ be a polynomial. If $\phi(\theta) = e^{P(\theta)}$ is the characteristic function of a real-valued random variable, then $\deg(P) \leq 2$. \end{theorem}

We also use the well known fact that $X_n \rightarrow X$ in distribution if and only if the associated characteristic functions $\phi_{X_n}(\theta) \rightarrow \phi_X(\theta)$ converge point-wise. Thus, to show the convergence $X_n^* \rightarrow Z$ in distribution, where $Z \sim \mathcal{N}(0,1)$, it suffices to show the point-wise convergence $P_n(e^{i\t/\s_n})e^{-i\theta \mu_n/\s_n} \rightarrow e^{-\t^2/2}$. With this target in mind, we seek an exponential form for the polynomials $P_n(z)$.

\subsection{An exponential form for $P$}

   To find such an exponential expression, we take logarithms of our probability generating function in an appropriate region. We use the principal branch of the logarithm: for $z \in \mathbb{C} \setminus \{0\}$, write $z = re^{i\theta}$, where $r >0$ and $\theta \in (-\pi,\pi]$. When then define $\log z = \log r + i\theta$. We use three simple properties of this function: that $\log z_1z_2 = \log z_1 + \log z_2 + 2\pi i t $, where $t \in \{-1,0,1\}$; that $\log e^z = z$, and that $-\log(1 - z) = \sum_k z^k / k$, for all $|z| \leq 1$ with $z\neq 1$.
	
	For a polynomial $P$ with roots $\{\zeta\}$, define 
	$$T_k = \sum_{\zeta : |\zeta| > 1} \zeta^{-k} \quad \mathrm{and} \quad S_k = \sum_{\zeta : |\zeta| < 1} \zeta^k.$$
The following lemma gives us our desired exponential form for our polynomials. 
	
	\begin{lemma} \label{lem:expFormOfP} Let $\delta > 0$, and let $X \in \{0,\ldots,n \}$ be a random variable with probability generating function $P$. If all the roots $\z$ of $P$ satisfy $|1 - \zeta| > \delta $ and $|\zeta| \neq 1$, then for all $z$ with $|z-1| < \delta$, we have the expression 
		\begin{equation} \label{equ:lemExpFormOfP}
				 P(z) = \exp\left( -\sum_{k \geq 1}\frac{T_k(z^k-1)}{k} - \sum_{k \geq 1}\frac{S_k(1/z^k-1)}{k} + R\log(z)   \right) \end{equation}
		where $R = |\{\zeta : |\zeta| < 1\}|.$ 
	\end{lemma}
	
	\begin{proof} Let $\{\zeta\}$ be the roots of $P(z)$. We write
		\begin{align*}
		P(z) &= \PP(X = n) \prod_{|\zeta| > 1}( z - \z) \prod_{|\zeta| < 1}(z - \zeta) \\		
		&= c z^R \prod_{|\zeta| > 1} (1 - z/\zeta) \prod_{|\zeta| < 1}(1 - \zeta / z),\end{align*}
		for some non-zero constant $c$.
		
		We now take a logarithm of this expression for $P(z)$ for $z$ satisfying $|z - 1| < \delta$.  This is possible as all the zeros $\z$ of $P$ satisfy $|1 - \zeta | > \delta$ and so we write 
		\[ \log P(z) =  \sum_{|\zeta| > 1} \log(1 - z/\zeta) + \sum_{|\zeta| < 1} \log(1 - \zeta/z) + R \log(z) + 2\pi i M(z) + \log c, \] 
		where $M(z)$ is an integer valued function. Now since $z$ satisfies $|z-1| < \delta$ and there are no zeros $\zeta$ with $|1 - \zeta| < \delta$, we may use the Taylor expansion of the logarithmic terms. We obtain
		\begin{align*}
		\log P(z) &=  \sum_{|\zeta| > 1}\left[ -\sum_{k\geq 1 } \frac{z^k}{k\zeta^k}\right] + \sum_{|\zeta| < 1}\left[ -\sum_{k\geq 1 } \frac{\zeta^k}{k z^k}\right] + R\log(z) + 2\pi i M(z) + \log c\\
		&= -\sum_{k \geq 1} \frac{T_kz^k}{k} -\sum_{k \geq 1} \frac{S_k}{z^k k} +R \log(z) + 2\pi i M(z) + \log c.
		\end{align*}
		Since $P(1) = 1$, we must have 
		\[  \sum_{k \geq 1} \frac{T_k}{k} + \sum_{k \geq 1} \frac{S_k}{k} = \log p_0 \mod 2\pi i, \] so we may write 
		\begin{align*}
		\log P(z) &=  -\sum_{k \geq 1}\frac{T_k(z^k-1)}{k} - \sum_{k \geq 1}\frac{S_k(1/z^k-1)}{k} + R\log(z)  + 2\pi i \widetilde{M}(z), 
		\end{align*}
		where $\widetilde{M}(z)$ is an integer-valued function. 
		We now exponentiate each side of the equation to obtain the desired result. 
	\end{proof}
	
	In the following lemma, we use the expression at (\ref{equ:lemExpFormOfP}) to obtain a exponential expression for $P(e^{i\theta})$. This lemma will ultimately be applied to the characteristic functions of our given sequence $\{X_n\}$. We shall also see that this expression gives us a way of writing the cumulants of a random variable $X$ in terms of the roots of its probability generating function. Recall that if $X$ is a random variable and we put $K(t) = \EE e^{tX}$ then the cumulant sequence $\{\kappa_n\}$ of $X$ is defined as the coefficients $K(t) = \sum_{k\geq 1} \frac{\kappa_k t^k}{k!}$ in the Talyor expansion of $K(t)$ about the origin.
	
	\begin{lemma}\label{lem:cumulant-form}  Let $\delta >0$, $n \in \mathbb{N}$ and let $X$ be a random variable taking values in $\{0,\ldots,n\}$ with probability generating function $P$. If all the roots $\z$ of $P$ satisfy both $|1- \zeta| > \delta $ and $|\zeta| \not=1$ then there exists an $\eps > 0$ so that for $\theta \in \mathbb{C}$ with $|\theta| < \eps$ we have 

		\begin{equation} \label{equ:LemCumulantForm}
				 P(e^{i\theta}) = \exp\left( -\sum_{m \geq 1} \frac{(i\theta)^m}{m!}(A_m + B_m) + Ri\theta \right ),\end{equation}
		\noindent 
		where $A_m = \sum_{k\geq 1 } T_k k ^{m-1}$, $B_m = (-1)^m\sum_{k \geq 1} S_k k^{m-1}$ and $R = |\{\zeta : |\zeta| < 1\}|$. Moreover, the $\eps$ may be chosen so that convergence of the double sum is uniform for $|\theta| < \eps$.
	\end{lemma}
 Anticipating the application of the Lemma~\ref{lem:cumulant-form}, the reader may feel that it is far too weak for our purposes; the neighbourhood on which we have equality depends dearly on the closet root of $P$ to the unit circle and, in general, will be far too small for our application. However, this statement \emph{will} be sufficient when used in tandem with the uniqueness of analytic continuations.  Indeed, in the course of the proof, we will see that both expressions for $P$ are analytic in suitably sized regions and thus, from this lemma, we will be able to conclude that they are equal in these larger domains.

	\begin{proof} The proof is straightforward; we use Lemma~\ref{lem:expFormOfP} to write $P(z)$ in an appropriate form, use the Taylor expansion of $e^{i\theta}$ and then exchange the order of summation. The only point at which we need to be careful is with the exchange of sums. However, we will quickly see that there is no danger as we are able to restrict $\theta$ to guarantee that the double sum is absolutely convergent.
	
	 Indeed, it is enough to show that the sums $\sum_m \sum_k \frac{\theta^m T_k k^{m-1}}{m!}$, $\sum_m \sum_k \frac{\theta^m S_k k^{m-1}}{m!}$, expanded from line (\ref{equ:LemCumulantForm}), are absolutely convergent.  We show the absolute convergence for the sum with the $T_k$ terms, and note that the proof of the other sum is analogous.  Define $\alpha = \min_{\zeta: |\zeta|>1} |\zeta| > 1$ and note that $|T_k| \leq \frac{n}{\alpha^k}$, which implies 
\[ \sum_{m \geq 1}\sum_{k \geq 1} \frac{|\theta|^m}{m!}|T_k|k^{m-1} \leq n \sum_{m \geq 1} \sum_{k \geq 1} \frac{|\theta|^m}{m!}\alpha^{-k} k^{m-1}
	 	\leq n\sum_{k \geq 1} \alpha^k \sum_{m \geq 1} \frac{|\theta|^m}{m!} k^{m-1}	,  \]
where the exchange in sums is allowed due to the positivity of the sequence. This last sum is bounded above by $ n\sum_{k \geq 1} \alpha^{-k}e^{k|\theta|} $, which converges absolutely whenever $|e^{|\theta|}/\alpha| < 1$.  Since $\alpha > 1$, there exists an $\eps > 0$ so that this occurs for $|\theta| < \eps$.  Free to exchange sums, we compute 
\begin{align*}
P(e^{i\theta}) &= \exp\left(- \sum_{k\geq 1} \frac{T_k (e^{i\theta k} - 1)}{k} - \sum_{k \geq 1} \frac{S_k (e^{-i\theta k}- 1)}{k} + Ri \theta \right) \\
&= \exp\left( -\sum_{k \geq 1} \sum_{m \geq 1} \frac{1}{k}\left(\frac{T_k (i k \theta)^m}{m!} + \frac{S_k (- i k \theta)^m}{m!} \right)  + Ri\theta \right) \\
&= \exp\left(- \sum_{m \geq 1} \frac{(i\theta)^m}{m!} \sum_{k \geq 1} \left(T_k k^{m-1} + (-1)^m S_k k^{m-1} \right) + Ri\theta\right) \,.
\end{align*}
\end{proof}
\begin{corollary}\label{cor:cumulants}
In the notation above, the $m$th cumulant of $X$ is $-(A_m + B_m)$ for $m \geq 2$ and $-A_1 - B_1 + R$ for $m = 1$\,.
\end{corollary}	
\begin{proof}
	Since $P(e^{i\theta}) = \EE[e^{i\theta X}]$, Lemma $\ref{lem:cumulant-form}$ implies that the moment generating function $\EE[e^{tX}]$ is analytic in some neighborhood of $t = 0$.  Taking logarithms and equating coefficients completes the proof.
\end{proof}

	\subsection{Controlling Higher Cumulants is Enough}
As noted earlier, we shall use Lemma~\ref{lem:cumulant-form} to obtain an expression for the characteristic function $\phi_n(\theta) = P_n(e^{i\theta})$ of $X_n$. So, in line with the strategy of showing convergence of the characteristic functions, we are led naturally to control the rescaled characteristic functions $\phi_n(\t/\s_n)$. It turns out this is somewhat tricky to do directly, and instead we will show that there is some appropriately chosen sequence $\{b_n\}$ for which we can control $\phi_n(\t/b_n)$. The following, slightly technical lemma, tells that controlling $\phi_n(\t/b_n)$ will be sufficient for our purposes. For this lemma, define the \emph{height} $h(P)$ of a polynomial $P(x) = \sum_{i=0}^d a_ix^i$ to be the magnitude of its largest coefficient $h(P) = \max_i\{ |a_i|\}$.  

For this lemma we require the following basic fact. If $P_n \in \mathbb{C}[x]$ is a sequence of polynomials of bounded degree that converge to a polynomial $P$ and $x_n$ is a sequence of complex numbers converging to $x$, then $\lim_n P_n(x_n) = P(x)$. We shall also use L\'{e}vy's continuity theorem, which says that if $\phi_n(\t)$ is a sequence of characteristic functions that converges point-wise to a function $\phi$ then $\phi$ is a characteristic function provided $\phi(\t)$ is continuous at $\t = 0$. See, for example, \cite[Theorem 3.3.6]{durrett} for a proof.
	
	\begin{lemma} \label{pr:high-cumulants-enough} 
		For each $n$, let $X_n$ be a real-valued random variable with mean $\mu_n$, standard deviation $\sigma_n < \infty$ and with characteristic function $\phi_n(\theta)$. Let $\{Q_n(\theta)\}$ be polynomials of degree at most $M \in \mathbb{N}$ with bounded height and for which no subsequence tends to the zero polynomial. Let $\{g_n(\theta)\}$ be a sequence of twice continuously differentiable functions with $|g_n^{(i)}(\theta)| = o(1)$ for $i \in \{0,1,2\}$ and all $\theta \in \R$. If there exists a sequence of positive real numbers $\{b_n\}$ so that
		$$\phi_n(\theta/b_n)e^{-i\mu_n \theta / b_n} = \exp\left(Q_n(\theta) + g_n(\theta) \right)\,,$$
		then $X_n$ satisfies a central limit theorem.
	\end{lemma}
	\begin{proof}
		We show that $\phi_n(\theta/\sigma_n)e^{-i\mu_n \theta / \s_n} \to e^{-\t^2/2}$ as $n \rightarrow \infty$ for all $\theta \in \mathbb{R}$. To show this, note that it is sufficient to show that for every infinite subsequence $S = \{n_k\}_k$ there is a further infinite subsequence $S' \subseteq S$ so that 
		$ \lim_{n \in S'}  \phi_n(\theta/\sigma_n)e^{-i\mu_n \theta / \s_n} = e^{-\t^2/2}$. We put $Z_n = \frac{X_n - \mu_n}{b_n}$ for each $n$ and let 
		$\psi_n$ denote the characteristic function of $Z_n$.
		
		Let $S$ be a given infinite subsequence. We now restrict to a subsequence $S'$ for which $Q_n(\t)$ converges to a polynomial $Q(\t)$ of degree at most $M$. Note that $Q$ cannot be the zero polynomial by the condition on $Q_n$ in the hypothesis. We have 
	$$\lim_{n \in S'} \phi_n(\theta/b_n)e^{-i\mu_n \t /b_n} = \exp\left(Q(\theta)\right)\, ,$$
	for each $\theta \in \R$. Put $\psi(\theta) = e^{Q(\theta)}$ so that $\psi_n \to \psi$ pointwise, and note that since $\psi_n$ is the characteristic function of $Z_n = \frac{X_n - \mu_n}{b_n}$ and the limit $\psi(\theta)$ is continuous at $\t = 0$, $\psi(\theta)$  is the characteristic function of some random variable $Y$. Thus, by Marcinkiewicz's Theorem (Theorem \ref{thm:MThm}), it follows that $Q(\theta) = c_0 + c_1\theta + c_2\theta^2$ is a polynomial of degree at most $2$, where we may assume that $c_0$ has imaginary part in the interval $[-\pi,\pi]$. We show that $c_0,c_1 = 0 $. It is easy to see that $c_0 = 0$ as $e^{c_0} = \psi(0) = \EE e^0 = 1$. To see that $c_1 = 0$, we need the following straightforward claim. 
\begin{claim}	
We have that $\lim_{n \in S'} \psi^{(i)}_n(\t) = \psi^{(i)}(\t)$, for $i \in \{0,1,2\}$.
\end{claim}
\noindent\emph{Proof of Claim :}
We observe that
\begin{align*}
\lim_{n \in S} \psi'_n(\t) &= \lim_{n\in S'} \frac{d}{d\t} e^{Q_n(\theta) + g_n(\theta)}\\
&= \lim_{n \in S'}(Q'_n(\theta) + g'_n(\theta))e^{Q_n(\theta) + g_n(\theta)}  \\
&= Q'(\theta)e^{Q(\theta)} = \psi'(\theta), \end{align*}
and similarly for $\psi^{(2)}(\theta)$. \qed

So to see that $c_1 = 0$, note that $\psi'(0) = c_1$ and, on the other hand, we have 
\[ 0 = \EE\left( \s_n^{-1}(X_n-\mu_n)\right) =  i\psi_n'(0) ,
\] which tends to $i\psi'(0) = ic_1$ as $n$ tends to infinity.

Thus we have shown that $Q(\theta) = c_2\theta^2$, and therefore $Z_n$ converges to a normal random variable $Z \sim \N(0,|c_2|^{-1})$ along the subsequence $S'$. We now need only to show that $X_n^* = \s_n^{-1}(X_n - \mu_n) $ (with scaling by $\s_n$) converges to a \emph{standard} normal along the subsequence $S'$. To this end, we note that
\begin{align}
 2c_2 &= \psi^{(2)}(0) = -\EE(Y^2)  \nonumber \\
 &= -\lim_{n\in S'} \EE\left( \frac{X_n - \mu_n}{b_n} \right)^2 = -\lim_{n\in S'} \left(\frac{\sigma_n}{b_n}\right)^2 , \label{eq:c_2Alpha} \end{align}
and thus the limit of the ratio $\s_n/b_n$ converges to some fixed number $\alpha \geq 0$, as $\s_n,b_n >0$.
Since $Q(\theta)= c_2\t^2$ is non-zero, we in fact have $\alpha > 0$. 

We now finish the proof of Lemma~\ref{pr:high-cumulants-enough}. Let $\t \in \mathbb{R}$ be fixed, put $\xi_n = b_n\t/\s_n$ and note that $\lim_{n \in S'} \xi_n = \alpha^{-1}\theta$.
We now write $\lim_{n \in S'} \phi_n(\t/\s_n)e^{-i\t\mu_n/\s_n}$ as 
\begin{align*} \lim_{n\in S'} \phi(\xi_n/b_n)e^{-\xi_n\mu_n/b_n} &= \lim_{n \in S'} \exp( Q_n(\xi_n) + o_n(1)) \\
&=  \exp\left(\lim_{n \in S'} Q_n(\xi_n)\right) = \exp(Q_n(\xi)),\\
&= \exp( c_2\alpha^{-2}\t^2 ) = \exp(-\t^2/2).\end{align*}
where the third to last equality follows from the fact that each $Q_n$ are polynomials and the last equality follows from (\ref{eq:c_2Alpha}). This completes the proof. \end{proof}
	
	\section{Proof of Theorem~\ref{thm:largeVarThm}}
	It will be convenient to assume that no roots of $P_n$ have modulus identically equal to $1$; we claim that this can be assumed without loss of generality by simply perturbing the random variables slightly. Indeed, for any $r$ close to $1$, we may define a random variable $\Xt_n$ via $\PP[\Xt_n = k] = C^{-1}r^k \PP[X_n = k]$ where $C = \sum_k r^k \PP[X_n = k]$. Then the probability generating function of $\Xt_n$ is $\Pt_n(z)=\EE[z^{\Xt_n}] = P_n(rz)P_n(r)^{-1}\,.$ And so we see that the roots of $\Pt_{\Xt_n}(z)$ are simply the roots of $P_n$, scaled by a factor of $1/r$. Also observe that for each fixed $n$, the functions $|\EE[X_n] - \EE[\Xt_n]|$, $|\Var[X_n] - \Var[\Xt_n]|$ and $\sup_A |\PP[X_n \in A] - \PP[\Xt_n \in A]|$ are continuous functions of $r$ and take a value of $0$ at $r = 1$.   Therefore, we may choose $r<1$ depending on $n$ and approaching $1$ so that 1) $X_n$ has a central limit theorem if and only if $\Xt_n$ does 2) none of the $\Pt_n$ have roots on the unit circle.
	\subsection{Proof of Theorem \ref{thm:largeVarThm}} \label{sec:ProofOfThmLargeVar}
	
	To control the higher cumulants, we bound the contribution from each root individually:
	\begin{lemma} \label{lem:cumulant-tails} For $\delta >0$, let $|\zeta| > 1$ satisfy $|1 - 1/\zeta| > \delta$.  Then for any $M \in \mathbb{N}$, we have \begin{equation} \label{eq:higher-cumulants-bound}
		\left|\sum_{m \geq M} \frac{(i \theta)^m}{m!} \sum_{k \geq 1} \zeta^{-k} k^{m-1} \right|
	\leq \frac{1}{2 \delta}\cdot  \frac{|\theta e / \delta|^M}{1 - |\theta e / \delta|},
		\end{equation}
		for all $\theta$ satisfying $|\theta e / \delta| < 1$.
	\end{lemma}
	\begin{proof}
		Note that \begin{equation*}
		\sum_{k\geq 1} k^{m-1} z^k = \sum_{0 \leq j \leq m-1} \frac{S(m-1,j) j! z^j}{(1 - z)^{j+1}}
		\end{equation*}
		where $S(n,k)$ are the Stirling numbers of the second kind.  To see the above, note that it holds for $m = 1$, and proceed by applying the operator $z\frac{d}{dz}$ to both sides and using the recurrence relation $S(n+1,k) = k S(n,k) + S(n,k-1)$ \cite[Equation (1.93)]{stanley}.  Applying this to the left-hand-side of \eqref{eq:higher-cumulants-bound} gives $$\sum_{m \geq M} \frac{(i \theta)^m}{m!} \sum_{k \geq 1} \zeta^{-k} k^{m-1} = \sum_{m \geq M} \frac{(i \theta)^m}{m!} \sum_{j = 0}^{m-1} \frac{S(m-1,j) j! (1/\zeta)^j}{(1 - 1/\zeta)^{j+1}}\,.$$ 
		Taking modulus of both sides and recalling $|\zeta| > 1$ and $|1 - 1/\zeta| > \delta$ gives an upper bound of \begin{equation}\label{eq:modulus-bound}
		\frac{1}{\delta}\sum_{m \geq M}\sum_{j = 0}^{m-1} \frac{|\theta/\delta|^m}{m!} S(m-1,j) j!\,.
		\end{equation}
		
		By \cite{rennie}, the upper bound of $S(n,k) \leq \frac{1}{2}\binom{n}{k} k^{n-k}$ holds for all $n$ and $k$.  Applying this bound to \eqref{eq:modulus-bound} gives a new upper bound of \begin{align*}
		\frac{1}{2\delta} \sum_{m \geq M} \sum_{j = 0}^{m-1} \frac{|\theta / \delta|^m}{m!} j! \binom{m}{j} j^{m-j} &= \frac{1}{2\delta} \sum_{m \geq M} \sum_{j = 0}^{m-1} |\theta / \delta|^m \frac{j^{m-j}}{(m-j)!} \\
		&\leq \frac{1}{2\delta} \sum_{m \geq M}  |\theta / \delta|^m \sum_{j = 0}^{m-1} \frac{m^{m-j}}{(m-j)!} \\
		&\leq \frac{1}{2\delta} \sum_{m \geq M}  |\theta e/ \delta|^m \\
		&= \frac{1}{2\delta} \frac{|\theta e / \delta|^M}{1 - |\theta e / \delta|}
		\end{align*}
		for $|\theta e / \delta| < 1$.
	\end{proof}
	
	\begin{corollary} \label{cor:cum-tails-less-than-1}
		Let $|\zeta| < 1$ satisfy $|1 - \zeta| > \delta$ for some $\delta$.  Then for any $M \in \mathbb{N}$, we have \begin{equation}
		\left| \sum_{m \geq M} \frac{(-i\theta)^m}{m!} \sum_{k \geq 1} \zeta^k k^{m-1} \right|\leq \frac{1}{2\delta}|\theta e / \delta|^M \frac{1}{1 - |\theta e / \delta|},
		\end{equation}
		for all $\theta$ satisfying $|\theta e / \delta| < 1$.
	\end{corollary}
	\begin{proof}
		Apply Lemma \ref{lem:cumulant-tails} to $1/\zeta$ and $-\theta$.  
	\end{proof}
\vspace{3mm}

\noindent With these preparations, we are in a position to prove Theorem~\ref{thm:largeVarThm}.

\vspace{3mm}

\noindent \emph{Proof of Theorem~\ref{thm:largeVarThm} :} For some $\eps>0$, let $\{X_n\}$ be a sequence of random variables for which $X_n$ takes values in $\{0,\ldots,n\}$ and has mean $\mu_n$ and standard deviation $\s_n > n^{\eps}$. Let $\{P_n\}$ be the corresponding sequence of probability generating functions. Put $\delta = \delta(n) = 1/\s_n^{1-\eps}$. By the discussion at the start of the present section, we may assume without loss that no $P_n$ has a root on the unit circle.
	
	We apply Lemma \ref{lem:cumulant-form} for each $P_n$ to find an $\eps' = \eps'(n) >0$ so that 
	\begin{equation}\label{eq:cumulant-form}
	P_n(e^{i\theta}) = \exp\left(-\sum_{m \geq 1} \frac{(i\theta)^m}{m!}(A_m + B_m) + Ri\theta \right) 
	\end{equation} for all $|\theta| < \eps'$. Put $F(\t) = -\sum_{m \geq 1} \frac{(i\theta)^m}{m!}(A_m + B_m)$ and recall that 
\[ A_m 	= \sum_{k\geq 1 } T_k k ^{m-1} = \sum_{\z, |\z| > 1} \sum_{k \geq 1} \z^{-k}k^{m-1}, 
\] and that 
\[ B_m =  (-1)^m\sum_{k \geq 1} S_k k^{m-1} = (-1)^m \sum_{\z, |\z| < 1} \sum_{k \geq 1} \z^{k}k^{m-1} ,
\] where the outer sums on the both right hand sides are over all roots $\z$ of $P_n$. 

\begin{claim} The equality at equation~(\ref{eq:cumulant-form}) holds for all $|\t| <  \delta/e$. \end{claim}
\noindent\emph{Proof of Claim : } To see this, we first show that $F(\t)$ exists and is complex analytic in the domain $|\t| < \delta/e$. 
As $F(\t)$ is defined by a power series, we need only to show that it is bounded in this region. We apply Lemma~\ref{lem:cumulant-tails} and Corollary~\ref{cor:cum-tails-less-than-1} to estimate the exponent in (\ref{eq:cumulant-form}). Indeed,
\begin{align*}
 F(\t) &\leq \left| \sum_{m \geq 1} \frac{(i\theta)^m}{m!}A_m \right|+\left| \sum_{m \geq 1} \frac{(i\theta)^m}{m!}B_m \right| \\
 &\leq \sum_{\z: |z|> 1} \left| \sum_{m \geq 1} \frac{(i\theta)^m}{m!}\sum_{k\geq 1} \z^{-k}k^{m-1} \right| + 
\sum_{\z: |z|< 1} \left| \sum_{m \geq 1} \frac{(i\theta)^m}{m!}\sum_{k\geq 1} \z^{k}k^{m-1} \right| \\
&\leq  \frac{n}{ \delta}\cdot   \frac{|\theta e / \delta|}{1 - |\theta e / \delta|}, \end{align*}
and therefore the right hand side of (\ref{eq:cumulant-form}) is a complex analytic function in the region $|\theta| <\delta/e$.
On the other hand, $P_n$ is a polynomial and thus $P(e^{i\theta})$ is an entire function of $\theta \in \mathbb{C}$. Thus, by the identity theorem for holomorphic functions, the equation at (\ref{eq:cumulant-form}) is valid for all $|\theta| < \delta/e$. \qed 

We now look to apply Lemma~\ref{lem:cumulant-tails} to control the $P(e^{i \theta / b_n})e^{-i \mu_n \theta / b_n}$, for some suitable sequence $\{b_n\}$. 
In particular, choose $b_n = \max_{2 \leq k \leq \lceil \eps^{-1} \rceil} \{ |A_k + B_k|^{1/k} \}$ and recall  that $|A_2+B_2|^{1/2} = \s_n$ from Corollary \ref{cor:cumulants}. 

\begin{claim} \label{claim:formforP} There exists a sequence of polynomials $\{Q_n\}$, and a sequence $\{g_n\}$ of twice continuously differentiable functions so that $0$ is not a limit point of $\{Q_n\}$, $g_n^{(i)}(\theta) = o(1)$ for all $\t \in \mathbb{R}$ and $i \in \mathbb{N} \cup \{0\}$ and
\[ P(e^{i \theta / b_n})e^{-i \mu_n \theta / b_n} = \exp\left( Q_n(\t) + g_n(\t) \right). \]\end{claim}
\noindent \emph{Proof of Claim : }
Put $M = 3\lceil \eps^{-2} \rceil$ and define 
\[Q_n(\theta) = \sum_{m=2}^M \frac{(i\t/b_n)^m}{m!}(A_m+B_m) \textit{ and  } g_n(\t) = \sum_{m > M }  \frac{(\t i/b_n)^m}{m!}(A_m+B_m). 
\]
Now note that the degree and height of the polynomials $\{Q_n(\theta)\}$ are bounded, simply by definition of the $b_n$. Moreover, $0$ cannot be a limit point of the $\{Q_n\}$ as there is always an $m$ for which the the $m$th term in $Q_n$ is at least $\frac{1}{m!}$, in absolute value.

By Lemma \ref{lem:cumulant-tails}, we see that 
\begin{align*}
 |g_n(\theta)| &= \left|\sum_{m \geq  M+1} \frac{i^m\theta^m}{b_n^mm!} (A_m +B_m) \right|  \\
 &\leq \frac{n}{\delta b_n^{M+1}} {|\theta e /  \delta|^{M+1}}  \\
&\leq |\theta e|^{M+1} \frac{\sigma_n^{\eps^{-1}} \sigma_n^{(1 - \eps)(M+2)}}{\sigma_n^{M+1}}  \end{align*}
since $b_n \geq \sigma_n \geq n^{\eps} $ and $\delta \geq \sigma_n^{-(1 - \eps)}.$  The total exponent on $\sigma_n$ is $$ \eps^{-1} + (1 - \eps)(M+2) - (M+1) \leq \eps^{-1} + 1 - M\eps < 0$$
since $M\eps > 3 \eps^{-1}$.  This shows $|g_n(\theta)| \to 0$, since $\sigma_n \to \infty$.

Moreover, this shows that for large enough $n$, $g_n(\theta)$ is an analytic function of $\theta$ in a domain containing the origin.
Since $g(\t)$ uniformly tends to zero in this domain, it follows that $g^{(i)}(\t)$ tends to zero, in this range for all $i \in \mathbb{N}$. This completes the proof of Claim~\ref{claim:formforP}. \qed
	
Thus Claim~\ref{claim:formforP} allows us to apply Lemma~\ref{pr:high-cumulants-enough} which, in turn, implies that the $\{X_n\}$ satisfy a central limit theorem. \qed
\vspace{5mm}	
	
We now quickly obtain our improvement on the Theorem of Ghosh, Liggett and Pemantle as a corollary. To do this we use the well known theorem of Cram\'er and Wold (see Corollary $6.3'$ in \cite{GLP}), which says a sequence of random variables $X_n = (X_n^{(1)},\ldots,X_n^{(d)}) \in \R^d$ converges in distribution to a centred Gaussian $Z \in \N(0,A)$ if and only if for each $a \in \mathbb{Q}^d$, the projections $\langle a , X_n \rangle $ converge in distribution to $\N(0,aAa^{T})$.
\vspace{2mm}

\emph{Proof of Corollary~\ref{cor:GLP} : } Let $\{Y_n\}$ be an appropriate sequence of random variables with probability generating functions $\{P_n(z_1,\ldots ,z_d)\}$. We let $A_n$ be the sequence of covariance matrices, let $\mu_n =(\mu_n^{(1)},\ldots,\mu_n^{(d)})$ be the sequence of means and put $Y_n^*  = s_n^{-1}( Y_n - \mu_n ) $. Recall that we have some sequence $s_n$ for which $s_n > n^{\eps}$ and $s_n^{-2}A_n \rightarrow A$. We now show $Y^*_n \rightarrow \N(0,A)$ by showing that we satisfy the conditions of the Cram\'er-Wold Theorem.

If $a\in \mathbb{Q}^d$ write $a = (a_1,\ldots,a_d)$ and observe that the projection of $Y_n$, $\langle a, Y_n \rangle$ has mean $\langle a , \mu_n \rangle$ and variance $a^TA_na$. Moreover, its probability generating function is exactly $P_{n,a}(z) = P_n(z^{a_1},\ldots,z^{a_d})$,
which has no zeros in a neighbourhood of $1$ due to the fact that $P_n$ is stable (\cite{GLP} Lemma 2.2). So if $a^TA_na > \deg(P_{n,a})^{\delta}$ for some $\delta >0$, we may apply our Theorem~\ref{thm:largeVarThm} to learn that $(a^TA_na)^{-1/2}\langle a , X_n - \mu_n \rangle \to \N(0,1)$. Therefore
\[ \langle a, Y_n^{*}\rangle = \frac{\langle a , X_n - \mu_n \rangle}{(a^TA_na)^{1/2}} \left(\frac{a^TAa}{s_n^2}\right)^{1/2} \rightarrow \N(0,a^TAa).
\] If $a^TAa$ is not growing polynomially in the degree, we certainly have $a^TAa = o(s^2_n)$ and thus $\langle a, Y^*_n \rangle$ converges to a point mass at $0$ by Chebyshev's inequality. This is the same as $\N(0,0) = \N(0, \lim_n (a^TAa)s_n^{-2}) $. Hence we finish by applying the Cram\'er-Wold theorem. \qed

\vspace{5mm}

In the following section, we turn to give some examples which show that polynomial growth condition on the variance cannot be be replaced with a logarithmic growth condition.

	\subsection{An example with logarithmic variance}
	
In this section we give class of examples that demonstrate the tightness of Theorems~\ref{thm:largeVarThm} and \ref{thm:TlToZero} in an appropriate sense. 
Let $T_m$ be a Poisson random variable with variance $1$, which has been conditioned on being at most $m$. Note that $\EE[T_m] = 1 + o_m(1)$ and $\Var[T_m] = 1 + o_m(1)$. It is also clear that $T_m$ has probability generating function $P_m(z)= \frac{1}{C_m}\sum_{k=0}^m \frac{z^k}{k!}$, where $C_m = \sum_{k=0}^m \frac{1}{k!}$. 

Curiously, the roots of these polynomials have received a considerable amount of attention going back to the work of Szeg\H{o} \cite{szegoExp}, who proved the remarkable fact that the roots of $P_m(mz)$ converge to the curve 
\[\{ z : |e^{-z}z| = 1, |z| \leq 1\}.\] We refer the reader to \cite{zemyan} for an exposition of these results and many further results.

 We shall only use the following consequence of the work of Szeg\H{o}:
\begin{lemma} \label{lem:szego} There exists a constant $c>0$ so that all of the roots $\z$ of $P_m(z)$ satisfy $cm \leq |\zeta| \leq  m\,$.
\end{lemma}
\noindent The following theorem contains our main construction.

\begin{theorem} \label{thm:logConstruction} For every $C >0$ there exists a sequence of random variables $\{X_n\}$ that does not satisfy a central limit theorem while $X_n \in \{0,\ldots,n\}$, $\s_n > C\log n$ and $|1-\z| > 1$ for all the roots $\z$ of the $\{P_n(z)\}$.
\end{theorem}

\begin{proof}
We first show that if we can find a example of a sequence that satisfies Theorem~\ref{thm:logConstruction} for \emph{some} value of $C >0$ then we may boost the variance (that is, increase $C$) by adding a collection of independent Bernoulli random variables to our example. So suppose that $\{X_n\}$ satisfies the statement in the theorem (with $\s_n = o(n^{1/2}))$ and let $C > 0$ be given. We produce a sequence $\{Y_n\}$ with standard deviation $\s_n(Y_n) > C\s_n$ which satisfies the statement of the theorem. For this, put $t(n) = \lceil (2C)^2\s_n^2 \rceil$ and let 
$S_n \sim \Bin(\frac{1}{2},t(n) )$ be independent of $X_n$. We define $Y_n = X_n + S_n \in \{0,\ldots,n +t(n)^2\} \subseteq \{0,\ldots, \{1+o_n(1))n \}$, and note that $\s_n(Y_n) = (C^2+1)^{1/2}\s_n +O(1) \geq C\s_n$. Moreover,
\[ Y_n^* = \frac{(X_n -\mu_n)+ (S_n - t(n)/2)}{(C^2\s_n^2 + \s_n)^{1/2} +O(1)} \]
 tends to a linear combination $\alpha Z + \beta W$, where $\alpha,\beta \not=0$, $Z \sim \N(0,1)$ and $W$ is distributed as the limit of $X_n^*$, which is not normal, by assumption. It follows that $\alpha Z + \beta W$ is not normal by Cram\'{e}r's Theorem \cite[Theorem $8.2.1$]{lukacs}.
 
It is also easy to calculate the roots of the probability generating function of $Y_n$. Indeed $Y_n$ has probability generating function $P_{Y_n}(z) = P_{X_n}(z)(z/2+1/2)^{t(n)}$, which only adds zeros at $z=-1$.

We now show that the theorem holds for the constant $C = 1$ and thus finish the proof of the theorem. We shall omit floors and ceilings in our discussion. We set $k = k(n) = \log n$ and let $T_m$ be a Poisson random variable which has been conditioned on being at most $m$. We define a random variable 
$X_n \in \{0,\ldots,n\}$ to be
\[ X_n = k T_{n/k} ,
\] a random variable with mean $k(1+o_n(1))$, standard deviation $k(1+o_n(1)) = (1+o(1))\log n$ and with probability generating function $P_{n/k}(z^k)$.
By Lemma~\ref{lem:szego}, the roots $\z$ of this polynomial satisfy $c^{1/k}(n/k)^{1/k} \leq |\z| \leq (n/k)^{1/k}$ and therefore tend uniformly to the circle of radius $e$. 

On the other hand, we have 
\[ X_n^* = \frac{X_n - \mu_n}{\s_n} = \frac{(T_{n/k} - 1+o(1))}{1+o(1)} ,
\] which converges in distribution to a random variable $T \sim \Pois(1) -1$. \end{proof}

\vspace{5mm}

With a different choice of the parameters in this construction, we demonstrate the sharpness of Theorem~\ref{thm:TlToZero}. That is, we show that 
there exists a sequence $\{X_n\}$ with $\s_n \rightarrow \infty$ that doesn't satisfy a central limit theorem and the roots of $P_n(z)$ \emph{just} fail to be polynomially large in modulus.

\begin{theorem} For any $\eps(n) \rightarrow 0 $, there exists a sequence $\{X_n\}$, of random variables that does not satisfy a central limit theorem where $X_n \in \{0,\ldots, n\}$, $\s_n \rightarrow \infty$ and the roots $\z$ of $P_n$ satisfy $|\z| > n^{\eps(n)}$. \end{theorem}

\begin{proof}
 Let $k = k(n) = 1/(2\eps(n))$ and define $X_n = k T_{n/k}$. \end{proof}

	\section{Proof of Theorem~\ref{thm:TlToZero}}
	
	In this section we prove Theorem~\ref{thm:TlToZero}, which roughly says that if the roots associated with our sequence $\{X_n\}$ tend to infinity at a rate polynomial in the degree of the $P_n$, then we have a central limit theorem, provided $\s_n \rightarrow \infty$. Note that Theorem~\ref{thm:TlToZero} also yields a central limit theorem when all roots converge to $0$ sufficiently quickly; simply consider the polynomial  $z^{\deg(P)}P_n(1/z)$ which is the probability generating function of $\deg(P) - X_n$.
	
If $P$ is a polynomial with $P(0) \not= 0$ and roots $\{\z\}$, we define 
	\[ T^*_{\ell} = \sum_{\z} |\z|^{-\ell},
	\] for each $\ell \in \mathbb{N}$.  We begin with a lemma that says that we have a slightly stronger form of Lemma~\ref{lem:cumulant-form} when $T_{\ell+1}^* \to 0$.
	\begin{lemma} \label{lem:expFormUsingTl}
		Let $\ell \in \mathbb{N}$, and for each $n$ let $X_n \in \{0,\ldots n\}$ be a random variable with probability generating function $P_n(z)$. If 
		$T^*_{\ell+1} \rightarrow 0$, then for all sufficiently large $n$ we may write 
		\[ P_n(e^{i\theta})e^{-i\theta \mu_n} = \exp\left( -\sum_{k=2}^\infty \frac{C_k(i\t)^k}{k!}  + h_n(\theta)   \right) \] 
		for all $\t \in \mathbb{R}$, where \[ C_k = T_1 + 2^{k-1}T_2 + \cdots + \ell^{k-1}T_{\ell} \]
		and $h_n^{(j)}(\theta) = o_n(1)$ for each fixed $\theta$ and $j \in \{0,1,2\}$.  Additionally we have, $C_2 =-\sigma_n^2 + o_n(1)$.
	\end{lemma}
	\begin{proof} Note that $T_{\ell+1}^* \to 0$ implies that $\alpha= \min_{\lambda} |\lambda| \to \infty$.  This gives the bound $|T_k| \leq \alpha^{-(k - \ell - 1)} T_{\ell + 1}^*$ for $k \geq \ell + 1$.
		We write $P_n$ exponentially by applying Lemma \ref{lem:expFormOfP} and using the identity $\mu_n = - \sum_{k \geq 1} T_k$. We obtain \begin{align}
		P_n(e^{i\theta})e^{-i \theta \mu_n} &= \exp\left(-\sum_{k \geq 1} \frac{T_k(e^{i\theta k} - 1 - i\theta k)}{k} \right) \\
		&= \exp\left(-\sum_{k =1}^{\ell} \frac{T_k(e^{i\theta k} - 1 - i \theta k)}{k} - \sum_{k = \ell + 1}^\infty \frac{T_k(e^{i\theta k} - 1 - i\theta k)}{k} \right) \\
		&= \exp\left(-\sum_{r = 2}^\infty \frac{C_k (i\theta)^k}{k!} + h_n(\theta) \right) \label{equ:expformTltozero}
		\end{align}
		where $h_n(\theta)$ is defined to be the tail sum in the line above.  Bound 
		\begin{align*}
		|h_n(\theta)| &\leq \sum_{k = \ell + 1}^\infty |T_k|\frac{|e^{i\theta k} - 1 - i\theta k|}{k} \\
		 &\leq |T_{\ell+1}^*| \sum_{k = \ell + 1}^\infty |\alpha|^{-(k - \ell - 1)} |e^{i\theta k}|  + |T_{\ell+1}^*|(1 + |\theta|) \sum_{k = \ell + 1}^\infty |\alpha|^{-(k - \ell - 1)} \\
		&\leq |e^{i\theta(\ell+1)}| \cdot |T_{\ell + 1}^*| \frac{1}{1 - | e^{i\theta} / \alpha|} + |T_{\ell+1}^*| (1 + |\theta|)\frac{1}{1 - |\alpha|}
		\end{align*}
		for $|e^{i\theta} / \alpha | < 1.$  Utilizing $|T_{\ell+1}^*| \to 0$ and $\alpha \to \infty$ shows that this tail sum converges to zero for each fixed $\theta$.  The first two derivatives of $h_n$ are bounded by $$|h_n'(\theta)| \leq \sum_{k = \ell+1}^\infty |T_k| \cdot |e^{i\theta k} - \theta|\qquad \mathrm{and}\qquad |h_n''(\theta)| \leq \sum_{k = \ell+1}^\infty |T_k| \cdot |k e^{i\theta k}|\,.$$
	
		Both sums can be shown to converge to zero for each $\theta$ by bounding $|T_k| \leq \alpha^{-(k - \ell - 1)}T_{\ell + 1}^*$ and summing the remaining series. To show that $\sigma_n^2 = - C_2 + o(1)$, we need only to take logarithms of both sides at (\ref{equ:expformTltozero}) and then use the definition of the cumulants and the fact that the variance of a random variable is its second cumulant.   
  \end{proof}
\vspace{2mm}	

At this point, the important observation is that the rate of growth of the large cumulants is ``linearly'' determined by the rate of growth of the terms $T_1,\ldots,T_{\ell}$. This will ultimately tells us that it is impossible for the cumulant sequence to grow fast enough to stop $\phi(\t/\s_n)$ from tending to a function of the form $e^{Q(\t)}$, where $Q$ is a polynomial. The following lemma captures the dependence between the growth of the cumulants and the growth of $T_1,\ldots,T_{\ell}$.
	
\begin{lemma} \label{lem:domTerm} Let $x_1,x_2,\ldots \in \mathbb{R}^{\ell}$ and let $y_1,y_2,\ldots \in \mathbb{R}^{\ell}$ be such that every set of $\ell$ consecutive vectors $y_{i+1},\ldots y_{i+\ell}$ are linearly independent and the $\{x_i\}$ are non-zero infinitely often. Then there exists an infinite set of integers $S \subseteq\mathbb{N}$ and a sequence of real numbers $\{B(m)\}_m$ so that the following hold:
\begin{enumerate}
\item \label{lem:domterm1}$\lim_{n \in S} \frac{1}{\|x_n\|}\left| \langle x_n,y_m \rangle - \|x_n\|\cdot B(m)\right| = 0$;
\item  \label{lem:domterm2} $B(m)$ is non-zero for infinitely many values of $m$;
\item  \label{lem:domterm3} $|B(m)| \leq \|y_m\|$.
\end{enumerate}
\end{lemma}	
\begin{proof}
Let $S \subseteq \mathbb{N}$ be an infinite set for which $\|x_n\| \neq 0$ for $n\in S$ and for which the limit $x_n/\|x_n\|$ 
converges to some non-zero vector $x \in \mathbb{R}^{\ell}$. In what follows, let us assume that we are only considering $n \in S$. Now note that there exists a sequence $c_n \in \mathbb{R}^{\ell}$ so that $x_n = x\|x_n\| + c_n$, where $c_n = o(\|x_n\|)$. 

Let $A_t$ be the $\ell \times \ell$ matrix with rows $y_{t + 1},\ldots,y_{t+ \ell}$; note that the linear map defined by $x \rightarrow A_tx$ has trivial kernel, as the vectors $y_{t+1},\ldots,y_{t+ \ell}$ are linearly independent, by assumption. We express
\begin{equation} \label{equ:innerproducts} A_tx_n = \|x_n\|(A_tx) + A_tc_n \end{equation} and note that $Ax$ is non-zero and $Ac_n = o(|x_n|)$. 

We now choose $B(m) = (A_{m-1}x)_1$ and note that $B(m)$ is non-zero for infinitely many values, as $A_tx$ is non-zero, for all $t$. This proves Item~\ref{lem:domterm2}. We also have that

\[|B(m)| = |(A_{m-1}x)_1| = |\langle x,y_m \rangle| \leq \|x\|\|y_m\| \leq \|y_m\|, \]
thus proving Item~\ref{lem:domterm3} in the lemma. \end{proof}	

\vspace{5mm}

\noindent We now combine these ingredients to prove Theorem~\ref{thm:TlToZero}.

\vspace{5mm}

\noindent\emph{Proof of Theorem~\ref{thm:TlToZero} : }
For $\ell \in \mathbb{N}$, let $X_n \in \{0,\ldots,n\}$ be a sequence of random variables for which $T^*_{\ell+1}$ converges to zero.  To show convergence to a random variable $Z \sim \N(0,1)$, it is sufficient to prove that every subsequence has a \emph{further} subsequence that converges to $\N(0,1)$, as we did in the proof of Lemma~\ref{pr:high-cumulants-enough}.
	
	We start by writing $P_n(e^{i\theta})$ in an exponential form; Lemma \ref{lem:expFormUsingTl} gives
	\begin{equation} \label{equ:charfun} P_n(e^{i\theta})e^{-i\theta \mu_n} =  \exp\left( \sum_{k=2}^\infty \frac{C_k(i\t)^k}{k!}  + h_n(\t)   \right) 
	\end{equation} where $C_k = T_1 + 2^{k-1}T_2 + \cdots + \ell^{k-1}T_{\ell}$. Here we have suppressed the explicit dependence on $n$ by writing $T_i = T_i(n)$.
	
 	Set $x_n = (T_1(n),\ldots,T_{\ell}(n))$ and put $y_m = (1^{m-1},2^{m-1},3^{m-1},\ldots,\ell^{m-1})$. Since $\s_n \rightarrow \infty$,
	and $\s_n^2(1+o(1)) = |C_2(n)| = |\langle x_n,y_1 \rangle| \leq \|x_n\|\|y_1\|$, we have that $\|x_n\| \rightarrow \infty$ and, in particular, $x_n$ is non-zero for infinitely many values. Also observe for every $t$, the $\ell$ consecutive vectors $y_{t+1},\ldots,y_{t+\ell}$ are linearly independent: simply consider the matrix with rows $y_{t+1},\ldots,y_{t+\ell}$; then the columns of this matrix are $v_1 = (1,\ldots,1), v_2 = 2^t(1,2,2^2,\ldots,2^{\ell-1}),\ldots, 
	v_t = \ell^t(1,\ell,\ell^2,\ldots,\ell^{\ell-1})$ which are linearly independent as $(1,\ldots,1),(1,2,2^2,\ldots,2^{\ell-1}),\ldots,(1,\ell,\ell^2,\ldots,\ell^{\ell-1})$ are the columns of a Vandermonde matrix of full rank. 
	 
	 With the conditions satisfied, we apply Lemma~\ref{lem:domTerm} to our $\{x_n\},\{y_m\}$  
	 to find an infinite subsequence of integers $S$ and a sequence of real numbers $\{B(m)\}$, so that 
	\begin{equation}\label{equ:Ck} |C_k - \|x_n\|\cdot B(k)| = \left| \langle x_n,y_k \rangle - \|x_n\|\cdot B(k)\right| = o(\|x_n\|) ,\end{equation}
	as $n \in S$ and tends to infinity, and $B(k)$ is non-zero for infinitely many values. Since we will only be considering $n \in S$ in what follows, we suppress the explicit dependence on $S$. 
	
	Now, let $k_0 \geq 2$ be the smallest value for which $B(k_0) \not= 0 $ and, anticipating an application of Lemma~\ref{pr:high-cumulants-enough}, we define 
	\begin{equation} \label{equ:bnChoice}
	 b_n = \max \{ |C_2(n)|^{1/2}, |C_3(n)|^{1/3},\ldots, |C_{k_0}(n)|^{1/k_0}\} 
	\end{equation} and split the sum at (\ref{equ:charfun}) according to $k_0$. That is, we set 
	\[ Q_n(\t) = \sum_{k=2}^{k_0} \frac{C_k(i\t/b_n)^k}{k!}, \]
	and put
	\[ g_n(\t) = \sum_{k > k_0} \frac{C_k(i\t/b_n)^k}{k!} .\]	
The following is trivial from the definition.
\begin{claim} $Q_n(\t)$ is a sequence of polynomials with bounded degree and height for which no subsequence converges to the $0$ polynomial.
\end{claim}
In the remainder of the proof, we only need to control the higher order terms. That is, we show the following.
\begin{claim} We have that $g_n^{(i)}(\t) = o_n(1)$ for all $\t \in \mathbb{R}$ and $i \in \mathbb{N}$. \end{claim}

\noindent\emph{Proof of Claim : }
	We start by observing that $|C_k|$ can be bounded using the Cauchy-Schwarz inequality,
	 \begin{equation} |C_k| = \langle x_n,y_k \rangle \leq \|x_n\|\|y_k\| \leq \ell^k \|x_n\|.  \label{equ:CkUpperBound} \end{equation}
	 From the choice of $b_n$ at (\ref{equ:bnChoice}) and equation~\eqref{equ:Ck}, we see that  
	\begin{equation}\label{equ:bnLowerbound} b_n \geq |C_{k_0} |^{1/k_0} =  (1+o(1))(\|x_n\|\cdot |B(k_0)|)^{1/k_0}\, \end{equation}	
	and so from lines (\ref{equ:CkUpperBound}) and (\ref{equ:bnLowerbound}) we have 
\[ \left|C_k b_n^{-k} \right| \leq \frac{\ell^{k} \|x_n\|(1+o_n(1))}{|B(k_0)|\|x_n\|^{k/k_0}} \, .
	\] Applying this estimate, we see that
	$$|g_n(\theta)| = \left|\sum_{k > k_0}^\infty \frac{C_k (i\theta)^k}{b_n^k k!}\right| \leq \sum_{k > k_0} \frac{(1+o_n(1))\ell^{k}|\t|^k}{k!|B(k_0)|\|x_n\|^{k/k_0-1}},$$ and that the right-hand-side of this inequality tends to zero as $n \rightarrow \infty$, due to the fact $\|x_n\| \rightarrow \infty$.
 Thus, $g_n$ is analytic and tends uniformly to zero in a neighbourhood of $\t$. It follows that $g_n^{(i)}(\t) = o_n(1)$, for all $i \in \mathbb{N}$. This completes the proof of the claim. \qed
\vspace{5mm}

\noindent We now finish the proof of Theorem~\ref{thm:TlToZero} by appealing to Lemma~\ref{pr:high-cumulants-enough}. \qed
	
	\section{Roots Avoiding a Forbidden Region}
	Let $P(z) = \EE[z^X]$ for some $X$ taking finitely many values in $\Z^{\geq 0}$.  Then since $P$ has real coefficients, we may factor $P$ into conjugate pairs
	$$P(z) = \prod_r \left(\frac{z + r}{1 + r}\right)\prod_\zeta\left( \frac{(z - \zeta)(z + \overline{\zeta})}{|1 + \zeta|^2}\right), $$
	where the first product is over the real roots and the second is over the non-real roots in the upper half plane. Note that each $r$ must be non-negative, as the coefficients of $P$ are non-negative.  For $\zeta \in \C\setminus \R$, define $P_\zeta$ and $p_i = p_i(\zeta)$ by 
	\begin{equation} \label{equ:Pzdef} P_\zeta(z) = \frac{(z - \zeta)(z + \overline{\zeta})}{|1 + \zeta|^2} = p_2 z^2 + p_1 z + p_0 .\end{equation}
	
	In the case of $\zeta = -r\in \R^{\leq 0}$, define $P_\zeta = \frac{z + r}{1 + r}$ and $p_0 = r/(1+r)$ and $p_1 = 1/(1+r)$, $p_2 = 0$, parallel to the above.  For each $\zeta \in \C\setminus \R^{>0}$, note that $P_\zeta(1) = 1$ and that $P$ is the product of the $P_\zeta$. We now proceed as if the $P_{\z}$ are probability generating functions. We define $\mu = \mu(\zeta) = 2 p_2 + p_1$ and for each $k \geq 0$ define $m_k = m_k(\zeta) = p_2(2 - \mu)^k + p_1(1 - \mu)^k + p_0(-\mu)^k$ and note that $m_0 = 1$ and $m_1 = 0$.  In the case where $P_\zeta(z) = \EE[z^X]$ for some $X$, note that $\mu = \EE[X]$ and $m_k = \EE[(X-\mu)^k]$.  Additionally, direct calculation shows that for all $\zeta \in \C\setminus \R^{>0}$ we have $P_{1/\zeta}(z) = z^2 P_{\zeta}(1/z)$, $\mu(\zeta) = 2 - \mu(1/\zeta)$, and $|m_k(\zeta)| = |m_k(1/\zeta)|$ for each $k$. 
	
The key property is that these ``central moments'' $m_k$ behave just as if the $P_{\z}$ are probability generating functions of independent random variables.

	\begin{lemma} \label{lem:mgf}
		Let $\zeta \in \C \setminus \R^{>0}$.  Then for any $\theta \in \C$ we have the expansion $$P_\zeta(e^{\theta})e^{-\theta \mu} = \sum_{k = 0}^\infty m_k \frac{\theta^k}{k!}\,.$$
	\end{lemma}
	\begin{proof}
	Simply write 
	\[ P(e^\t)e^{-\mu\t} = p_2e^{\t(2-\mu)} + p_1e^{\t(1-\mu)} + p_0e^{-\mu} = \sum_{k\geq 0} \frac{\t^k}{k!}\left(p_2 (2-\mu)^k + p_1(1-\mu)^k + p_0\mu^k\right),\] where we have used the Talyor expansion of $e^{\t}$.
	\end{proof}
	
	Writing $P$ as a product over the $P_\zeta$'s will give an expansion for the moments of $X_n$ in terms of the $m_k$'s; so to show that our random variable is normal, we only need to control the $m_k$'s. The following lemma gives us this control. It is also the point in the proof where we make use of the fact that the roots $\z$ avoid the region $S$. 
	
	\begin{lemma} \label{lem:moments}
		Let $\delta > 0$.  Then for $k \geq 2$ there exists a constant $c_k > 0$ so that for all $\zeta$ satisfying $|1 - \zeta| > \delta$,  $|m_k| \leq c_k$.  Furthermore, if $d(S,\zeta) > \delta$ then there is a constant $c_k'(\delta) > 0 $ such that $|m_k| \leq c_k'(\delta) m_2$. 		
	\end{lemma}
	\begin{proof}
	In the case of $\zeta \in \R^{\leq 0}$, we have that $P_\zeta(z) = \EE[z^X]$ where $X$ is a Bernoulli random variable in which case $|m_k| \leq m_2 \leq 1$.  We now deal with the case when $\zeta \in \C \setminus \R$.
		
Fix $k \geq 2$.  One can see directly from the definition at \eqref{equ:Pzdef} that as  $|\zeta| \to \infty$, we have $p_2, p_1 \to 0$ and $p_0 \to 1$, and therefore $m_k \rightarrow 0 $ as $\z \to \infty$. Now $m_k$ is a continuous function of $\z$, in the region $|z-1| > \delta$, and thus $m_k$ attains a maximum. That is, $|m_k| \leq c_k$, for some constant $c_k$.
		
We now turn to show the second part of the lemma, that $|m_k| < c'_k(\delta)m_2$, provided $d(S,\zeta) > \delta$.  Write $\z = a + ib$. Since $|\z - 1| > \delta$, a direct calculation gives
$$m_k =  \frac{(2(a^2 + b^2) - 2a)^k - 2a(a^2 + b^2 - 1)^k + (a^2 + b^2)(2a - 2)^k}{((1 - a)^2 + b^2)^{k+1}}\,.$$

To see that $m_2 \leq 0 $ if and only if $(a+ib) \in S$, put $r = (a^2 + b^2)^{1/2}$ and factor
\[ m_2 = \frac{ -2((a-1)^2 + b^2) (a (r^2 + 1) - 2 r^2)}{((1 - a)^2 + b^2)^{3}} = \frac{-2(a (r^2 + 1) - 2 r^2)}{((1 - a)^2 + b^2)^{2}}. \]

As the denominator is always positive, the sign depends only on the numerator, which is the same expression that appears in the definition of $S$.
		
So write 
$$\frac{m_k}{m_2} = \frac{(2(a^2 + b^2) - 2a)^k - 2a(a^2 + b^2 -1)^k + (a^2 + b^2)(2a - 2)^k}{(1 - 2a + a^2 + b^2)^{k-2}((2a^2 + 2b^2 - 2a)^2 - 2a(a^2 + b^2 - 1)^2 + (a^2 + b^2)(2a - 2)^2)}\,.$$

Since the line $\{2 + ib : b \in \R\}$ is contained in $S$, we have that $a < 2 - \delta$ for all $a + ib$ with $d(S,a+ib) > \delta$.  
We bound $|m_k|/m_2$ in two steps; first for $a \in (-2,2-\delta)$ and $b$ sufficiently large, and then for $|a^2 + b^2|$ sufficiently large and $a < -1$.  Since the remaining set $\{\zeta \in \C: d(S,\zeta) \geq \delta \}$ without these two regions is compact, the continuous function $|m_k|/m_2$ achieves a maximum on it.   

We first bound $|m_k / m_2|$ for $a \in [-2,2-\delta]$ and $b$ large.  Divide the numerator and denominator of the expression for $m_k / m_2$ by $b^{2k}$ and bound the numerator by $(12)^k + 4 + 5\cdot 2^k$ for all $a \in [-2,2]$, $b \geq 1$.  The denominator requires a bit more care.  We have $(1 - 2a + a^2 + b^2)^{k-2} \geq b^{2k - 4}$.  Since $[-2,2-\delta]$ is a compact set, we can find $M_\delta$ so that for all $b > M_\delta$ and $a \in [-2,2-\delta]$ we have $$\left|\frac{(2a^2 + 2b^2 - 2a)^2}{b^4} - 4 \right| + \left|\frac{-2a(a^2 + b^2 - 1)^2}{b^4} - (-2a)  \right| + \left|\frac{(a^2 + b^2)(2a - 2)^2}{b^4} \right| < \delta\,.$$
Then for $a \in [-2,2+\delta]$ and $b > M_\delta$, $$\left|\frac{m_k}{m_2}\right| \leq \frac{12^k + 4 + 5\cdot 2^k}{4a - 2 - \delta} \leq \frac{12^k + 4 + 5\cdot 2^k}{ \delta}\,.$$
We now bound $|m_k / m_2|$ for large $r^2 = a^2 + b^2$ and $a \leq -1$. So if we write $a + ib = r e^{i\theta}$, we bound
 
$$\frac{m_k}{m_2} = \frac{(2r^2- 2r \cos(\theta))^k - 2r\cos(\theta)(r^2 -1)^k + r^2(2r\cos(\theta) - 2)^k}{(1 - 2r\cos(\theta) + r^2)^{k-2}((2r^2 - 2r\cos(\theta))^2 - 2r\cos(\theta)(r^2 - 1)^2 + r^2(2r\cos(\theta) - 2)^2)}\, ,$$
for sufficiently large $r$ and $r \cos(\theta) < - 1$.   Divide the numerator and denominator of the expression of $m_k / m_2$ by $r^{2k + 1}$.  Then for $\cos(\theta) < 0$, the numerator is bounded above by $$\frac{4^k}{r} - 2\cos(\theta) + \frac{1}{r}$$ for $r \geq 4$.  Similarly, the denominator over $r^{2k+1}$ is bounded below by $-2\cos(\theta)(1 - \frac{1}{r^2})^2  \geq -\cos(\theta)$ for $r \geq 4$ and $\cos(\theta) < 0$.  Thus, for $r \geq 4$ and $-r \cos(\theta) > 1$, we have $$\frac{m_k}{m_2} \leq \frac{\frac{4^k}{r} - 2 \cos(\theta) + \frac{1}{r}}{-\cos(\theta)} = \frac{4^k}{-r\cos(\theta)} + 2 + \frac{1}{-r\cos(\theta)} \leq 4^k + 2 + 1\,.$$
Thus, we have that $|\frac{m_k}{m_2}|$ is bounded above for all $r \geq 4$ and $r \cos(\theta) < -1$.
	\end{proof}

	\emph{Proof of Theorem \ref{thm:region} :}  The proof is by the method of moments; Carleman's continuity theorem tells us that convergence $\frac{X_n - \mu_n}{\sigma_n} \to \N(0,1)$ in distribution follows from the convergence of each moment. That is, it is sufficient to show that for each $k$ we have
	$$\EE\left[\left(\frac{X_n - \mu_n}{\sigma_n}\right)^k\right] \xrightarrow{n\to\infty} \begin{cases}
	(k - 1)!! & \text{if }k \text{ is even}; \\
	0 & \text{if }k \text{ is odd },
	\end{cases}\,$$
	since the right-hand-side is the $k$th moment of a standard normal.  Define $M_k$ to be the $k$th moment of $\frac{X_n - \mu_n}{\sigma_n}$, i.e. the left-hand-side of the above equation.  Since each $X_n$ is bounded, the expansion $$\EE\left[\exp\left(i\theta(X_n - \mu_n)/\sigma_n\right)\right] = P(e^{i\theta / \sigma_n})e^{-i\theta\mu_n / \sigma_n} = \sum_{k = 0}^\infty \frac{M_k (i\theta)^k}{k!}\,$$
	converges in a neighborhood of $\theta = 0$.
	
	By Lemma \ref{lem:mgf} and noting that $\mu_n = \sum \mu(\zeta)$, we have
	\begin{equation} \label{eq:moment-expansion}
	P(e^{i\theta/\sigma_n})e^{-i\theta \mu_n / \sigma_n} =\prod_\zeta P_\zeta(e^{i\theta/\sigma_n})
	e^{-i\theta \mu(\zeta) / \sigma_n}= \prod_\zeta \left(\sum_{k = 0}^\infty \frac{m_k(\zeta)(i\theta)^k}{\sigma_n^k k!} \right)\,.\end{equation}
	
	Since $m_1(\zeta) = 0$ for each $\zeta$, this means that $\sum_\zeta m_2(\zeta) = \sigma_n^2$.  Expanding this product and equating coefficients---i.e. using the identity theorem---gives
	
	\begin{align}M_k &= \frac{k!}{\sigma_n^k}\sum_{j_1 + \cdots + j_l = k} \sum_{i_1 < i_2 < \ldots < i_l} \frac{m_{j_1}(\zeta_{i_1})  m_{j_2}(\zeta_{i_2} ) \cdots m_{j_l}(\zeta_{i_l}) }{j_1! j_2! \cdots j_l!} \nonumber\\
	&= \frac{1}{\sigma_n^k}\sum_{j_1 + \cdots + j_l = k} \binom{k}{j_1,\ldots,j_l}\frac{1}{l!} \sum_{i_1,\ldots,i_l~\mathrm{distinct}} m_{j_1}(\zeta_{i_1})\cdots m_{j_l}(\zeta_{i_l}) \nonumber  \\ 
	&= \frac{1}{\sigma_n^k}\sum_{j_1 + \cdots + j_l = k} \binom{k}{j_1,\ldots,j_l}\frac{1}{l!} S(j_1,\ldots,j_l) \label{eq:moment-equal}
	\end{align}
	where $S(j_1,\ldots,j_l) = \sum m_{j_1}(\zeta_{i_1})\cdots m_{j_l}(\zeta_{i_l})$ is the innermost sum in the line above. Since $m_1(\zeta) = 0$ for all $\zeta$, it is sufficient to consider only compositions $(j_1,\ldots,j_l)$ so that each part is at least $2$.  
	
	\begin{claim} \label{claim:BoundOnmoments}
		Let $j_1 + \ldots + j_l = k$ where $j_1,\ldots,j_l \geq 2$ and $j_i > 2$, for some $i\in [l]$.  Then $S(j_1,\ldots,j_l) = o(\sigma_n^k)$. 
	\end{claim}
\noindent	\emph{Proof of Claim :} Define $\NV = \{z \in \C: d(S,z) \leq \delta \}$ and $\PV = \C\setminus \NV$ and recall that $|\NV| = N_n = N_n(\delta)$.  If we write 
$\sigma_n^2 = \sum_{\zeta \in \PV}m_2(\zeta) + \sum_{\zeta \in \NV}m_2(\zeta), $ it follows that 
\[ \sum_{\zeta \in \PV} m_2(\zeta) \leq \sigma_n^2 + c'_2 N_n = o(\sigma_n^3)\, ,\]
from Lemma~\ref{lem:moments} and from the assumption about the number of zeros $\z$ within $\delta$ of our region $S$. From Lemma~\ref{lem:moments} we also obtain an estimate for $j > 2$, 
\begin{equation} \label{equ:boundonSumofmk}
\sum_{\zeta} |m_j(\zeta)| \leq \sum_{\zeta \in \PV} |m_j(\zeta)| + \sum_{\zeta \in \NV} |m_j(\zeta)| \leq c_j(\delta) \sum_{\zeta \in \PV} m_2(\zeta) + c_j'N_n = o(\sigma_n^3) = o(\sigma_n^j)\,.\end{equation}
		
 Since $S(j_1,\ldots,j_l)$ is the same for all permutations of $j_1,\ldots,j_l$, we assume without loss of generality that $j_1,\ldots, j_l$ is of the form $j_1,\ldots, j_r,2,2,\ldots, 2$ with each $j_1,\ldots,j_r > 2$. 
\begin{align}
	S(j_1,\ldots,j_l) &= \sum_{i_1,\ldots,i_l~\mathrm{dist.}} m_{j_1}(\zeta_{i_1}) \cdots m_{j_r}(\zeta_{i_r})m_2(\zeta_{i_{r+1}})\cdots m_2(\zeta_{i_l}) \nonumber \\
	&=\sum_{i_1,\ldots,i_{l-1}~\mathrm{dist.}} m_{j_1}(\zeta_{i_1}) \cdots m_{j_r}(\zeta_{i_r})m_2(\zeta_{i_{r+1}})\cdots m_2(\zeta_{i_{l-1}})  \sum_{i_l \neq i_1,\ldots,i_{l-1}} m_2(\zeta_{i_l}) \nonumber \end{align}
Now $ \sum_{i_l \neq i_1,\ldots,i_{l-1}} m_2(\zeta_{i_l})  = \sum_{\z} m_2(\z) + O(1) = \s_n^2 + O(1)$. Thus, applying this $r-l$ times yields
	\begin{equation}
	 =  \sum_{i_1,\ldots,i_{r}~\mathrm{dist.}} m_{j_1}(\zeta_{i_1}) \cdots m_{j_r}(\zeta_{i_r})(\sigma_n^2 + O(1))^{l-r}\,. \label{eq:Sj-2s-out} \end{equation}
Applying the triangle inequality, we see 
\begin{align*} |S(j_1,\ldots,j_l)| &\leq  (\sigma_n^2 + O(1))^{l-r} \sum_{i_1,\ldots,i_r~\mathrm{dist.}} |m_{j_1}(\zeta_{i_1})| \cdots |m_{j_r}(\zeta_{i_r})| \\
&\leq (\sigma_n^2 + O(1))^{l-r} \prod_{i=1}^r \left( \sum_{\z} |m_{j_i}(\z)| \right) = o(\sigma_n^k),
\end{align*}
where the last equality follows from line~(\ref{equ:boundonSumofmk}). This proves the claim. \qed

\vspace{5mm}

We now apply Claim~\ref{claim:BoundOnmoments} to \eqref{eq:moment-equal} to learn that  $$M_k = \begin{cases}
(k-1)!! \frac{S(2,2,\ldots,2)}{\sigma_n^k} + o(1) &\text{ if }k \text{ is even} ;\\
o(1) &\text{ if }k \text{ is odd}.
\end{cases}\,.$$
	
\noindent	Arguing as in \eqref{eq:Sj-2s-out} shows that $S(2,2,\ldots,2) = (\sigma_n^2 + O(1))^{k/2}$, thereby implying $$\sigma_n^{-k}S(2,2,\ldots,2) \to 1\,.$$ Applying Carleman's continuity theorem (see, for instance, \cite[Theorems $30.1$ and $30.2$]{billingsley}) completes the proof. \qed

\section{Acknowledgements}
We would like to thank Robin Pemantle for bringing to our attention this circle of questions and B\'{e}la Bollob\'{a}s and Robert Morris for comments. The second named author would like to thank Robin Pemantle and the University of Pennsylvania for hosting him while this research was conducted.  

	\bibliographystyle{plain}
	\bibliography{BibRootsCLTs}

\begin{thebibliography}{10}

\bibitem{nonNegFactor}
R.W. Barnard, W.~Dayawansa, K.~Pearce, and D.~Weinberg.
\newblock Polynomials with non-negative coefficents.
\newblock {\em Proc. Amer. Math. Soc.}, 113:77–85, 1991.

\bibitem{BorweinPowerSeries}
F~Beaucoup, P~Borwein, David Boyd, and Christopher Pinner.
\newblock Multiple roots of [-1, 1] power series.
\newblock 57:135--147, 02 1998.

\bibitem{billingsley}
P.~Billingsley.
\newblock {\em Probability and measure}.
\newblock Wiley Series in Probability and Mathematical Statistics. John Wiley
  \& Sons, Inc., New York, third edition, 1995.
\newblock A Wiley-Interscience Publication.

\bibitem{Bilu}
Y.~Bilu.
\newblock Limit distribution of small points on algebraic tori.
\newblock {\em Duke Math J.}, 89(3), 1997.

\bibitem{BlochPolya}
A.~Bloch and G.~P\'{o}lya.
\newblock On the roots of certain algebraic equations.
\newblock {\em Proc. London Math. Soc.}, 33:102--114, 1932.

\bibitem{bbl}
J.~Borcea, P.~Br\"and\'en, and T.~M. Liggett.
\newblock Negative dependence and the geometry of polynomials.
\newblock {\em J. Amer. Math. Soc.}, 22(2):521--567, 2009.

\bibitem{borwein-erdelyi}
P.~Borwein and T.~Erd\'elyi.
\newblock {\em Polynomials and polynomial inequalities}, volume 161 of {\em
  Graduate Texts in Mathematics}.
\newblock Springer-Verlag, New York, 1995.

\bibitem{cauchy1828exercises}
A.~L.~B. Cauchy.
\newblock {\em Exercises de math{\'e}matiques}, volume~3.
\newblock Bure fr{\`e}res, 1828.
\newblock p. 122.

\bibitem{durrett}
R.~Durrett.
\newblock {\em Probability: theory and examples}, volume~31 of {\em Cambridge
  Series in Statistical and Probabilistic Mathematics}.
\newblock Cambridge University Press, Cambridge, fourth edition, 2010.

\bibitem{Erdelyi}
T.~Erd\'{e}lyi.
\newblock On the zeros of polynomials with {L}ittlewood-type coefficient
  constrains.
\newblock {\em Michigan. Math. J.}, 49:97--111, 2001.

\bibitem{ErdosOfford}
P.~Erd\H{o}s and A.~C. Offord.
\newblock On the number of real roots of a random algebraic equation.
\newblock {\em Proceedings of the London Mathematical Society},
  s3-6(1):139--160.

\bibitem{erdos-turan}
P.~Erd\H{o}s and P.~Tur\'an.
\newblock On the distribution of roots of polynomials.
\newblock {\em Ann. of Math. (2)}, 51:105--119, 1950.

\bibitem{eremenko}
A.~Eremenko and A.~Fryntov.
\newblock Remarks on the {O}brechkoff inequality.
\newblock {\em Proc. Amer. Math. Soc.}, 144(2):703--707, 2016.

\bibitem{Ganelius}
T.~Ganelius.
\newblock Sequences of analytic functions and their zeros.
\newblock {\em Ark. Mat.}, 3:1--50, 1954.

\bibitem{GLP}
S.~Ghosh, T.~M. Liggett, and R.~Pemantle.
\newblock Multivariate {CLT} follows from strong {R}ayleigh property.
\newblock In {\em 2017 {P}roceedings of the {F}ourteenth {W}orkshop on
  {A}nalytic {A}lgorithmics and {C}ombinatorics ({ANALCO})}, pages 139--147.
  SIAM, Philadelphia, PA, 2017.

\bibitem{AGran}
A.~Granville.
\newblock The distribution of roots of a polynomial.
\newblock In {\em Equidistribution in Number Theory, an Introduction.}, pages
  93--102, 2007.

\bibitem{hwang-zacharovas}
H-K. Hwang and V.~Zacharovas.
\newblock Limit distribution of the coefficients of polynomials with only unit
  roots.
\newblock {\em Random Structures Algorithms}, 46(4):707--738, 2015.

\bibitem{kac1943}
M.~Kac.
\newblock On the average number of real roots of a random algebraic equation.
\newblock {\em Bull. Amer. Math. Soc.}, 49(4):314--320, 04 1943.

\bibitem{kac-book}
M.~Kac.
\newblock {\em Probability and related topics in physical sciences}, volume
  1957 of {\em With special lectures by G. E. Uhlenbeck, A. R. Hibbs, and B.
  van der Pol. Lectures in Applied Mathematics. Proceedings of the Summer
  Seminar, Boulder, Colo.}
\newblock Interscience Publishers, London-New York, 1959.

\bibitem{LPRS}
J.~L. Lebowitz, B.~Pittel, D.~Ruelle, and E.~R. Speer.
\newblock Central limit theorems, {L}ee-{Y}ang zeros, and graph-counting
  polynomials.
\newblock {\em J. Combin. Theory Ser. A}, 141:147--183, 2016.

\bibitem{LittlewoodOfford}
J.~E. Littlewood and A.~C. Offord.
\newblock On the number of real roots of a random algebraic equation.
\newblock {\em Journal of the London Mathematical Society}, s1-13(4):288--295.

\bibitem{LittlewoodOfford-II}
J.~E. Littlewood and A.~C. Offord.
\newblock On the number of real roots of a random algebraic equation. {II}.
\newblock In {\em Mathematical Proceedings of the Cambridge Philosophical
  Society}, volume~35, pages 133--148. Cambridge University Press, 1939.

\bibitem{LittlewoodOfford-III}
J.~E. Littlewood and A.~C. Offord.
\newblock On the number of real roots of a random algebraic equation. {III}.
\newblock {\em Rec. Math. [Mat. Sbornik] N.S.}, 12(54):277--286, 1943.

\bibitem{lukacs}
E.~Lukacs.
\newblock {\em Characteristic functions}.
\newblock Hafner Publishing Co., New York, 1970.
\newblock Second edition, revised and enlarged.

\bibitem{Marcin}
J.~Marcinkiewicz.
\newblock Sur une propriete de la lot de {G}auss.
\newblock {\em Math. Zeits.}, 44:612--618, 1938.

\bibitem{Mignotte}
M.~Mignotte.
\newblock Remarque sur une question relative \'{a} des fonctions
  conjugu\'{e}es.
\newblock {\em C.R. Acad. Sci. Paris S\'{e}r. I. Math.}, 315(8), 1992.

\bibitem{obrechkoff}
N.~{Obrechkoff}.
\newblock {Sur un probl\`eme de Laguerre.}
\newblock {\em {C. R. Acad. Sci., Paris}}, 177:102--104, 1923.

\bibitem{OP}
A.M. Odlyzko and B.~Poonen.
\newblock Zeros of polynomials with 0,1 coefficients.
\newblock {\em L'Enseign. Math.}, 39:317--348, 1993.

\bibitem{offord}
A.~C. Offord.
\newblock The distribution of the values of an entire function whose
  coefficients are independent random variables.
\newblock {\em Proc. London Math. Soc. (3)}, 14a:199--238, 1965.

\bibitem{pemantle}
R.~Pemantle.
\newblock Personal Communication.

\bibitem{pemantle-survey}
R.~Pemantle.
\newblock Hyperbolicity and stable polynomials in combinatorics and
  probability.
\newblock In {\em Current developments in mathematics, 2011}, pages 57--123.
  Int. Press, Somerville, MA, 2012.

\bibitem{peres-virag}
Y.~Peres and B.~Vir\'ag.
\newblock Zeros of the i.i.d.\ {G}aussian power series: a conformally invariant
  determinantal process.
\newblock {\em Acta Math.}, 194(1):1--35, 2005.

\bibitem{rennie}
B.~C. Rennie and A.~J. Dobson.
\newblock On {S}tirling numbers of the second kind.
\newblock {\em J. Combinatorial Theory}, 7:116--121, 1969.

\bibitem{Schur}
I.~Schur.
\newblock Untersuchungen \"{u}ber algebraische gleichungen.
\newblock {\em Sitz. Preuss. Akad. Wiss., Phys.- Math. Kl.}, pages 403--428,
  1933.

\bibitem{soshnikov}
A.~Soshnikov.
\newblock Gaussian limit for determinantal random point fields.
\newblock {\em Ann. Probab.}, 30(1):171--187, 2002.

\bibitem{Sound}
K.~Soundararajan.
\newblock Equidistribution of roots of polynomials.
\newblock {\em American Math Monthly}.

\bibitem{stanley}
R.~P. Stanley.
\newblock {\em Enumerative combinatorics. {V}olume 1}, volume~49 of {\em
  Cambridge Studies in Advanced Mathematics}.
\newblock Cambridge University Press, Cambridge, second edition, 2012.

\bibitem{szegoExp}
G.~Szeg\H{o}.
\newblock \"{U}ber eine eigenschaft der exponentialreihe.
\newblock {\em Sitzungsber Berliner Math. Gesellschaft}, 23:50--64, 1924.

\bibitem{Szego}
G.~Szeg\H{o}.
\newblock Bemerkungen zu einem satz von {E}. {S}chmidt \"{u}ber algebraische
  gleichungen.
\newblock {\em Sitz. Preuss. Akad. Wiss., Phys.- Math. Kl.}, pages 86--98,
  1934.

\bibitem{zemyan}
S.~M. Zemyan.
\newblock On the zeroes of the {$N$}th partial sum of the exponential series.
\newblock {\em Amer. Math. Monthly}, 112(10):891--909, 2005.

\end{thebibliography}
	
\end{document}